\documentclass[]{class2}

\usepackage{graphicx}
\usepackage{amsmath, amsthm, amssymb, xfrac}

\begin{document}

\begin{frontmatter}   

\titledata{An equivalent formulation of the \mbox{Fan-Raspaud} Conjecture and related problems}           
{}                 

\authordata{Giuseppe Mazzuoccolo}            
{Dipartimento di Informatica, Universit\`{a} di Verona, Strada Le Grazie 15, Verona, Italy}     
{giuseppe.mazzuoccolo@univr.it}                     
{}          

\authordata{Jean Paul Zerafa}            
{Dipartimento di Scienze Fisiche, Informatiche e Matematiche,\\ Universit\`{a} di Modena e Reggio Emilia, Via Campi 213/B, Modena, Italy} 
{jeanpaul.zerafa@unimore.it}
{}                                       


\begin{abstract}
In 1994, it was conjectured by Fan and Raspaud that every simple bridgeless cubic graph has three perfect matchings whose intersection is empty. In this paper we answer a question recently proposed by Mkrtchyan and Vardanyan, by giving an equivalent formulation of the \mbox{Fan-Raspaud} Conjecture. We also study a possibly weaker conjecture originally proposed by the first author, which states that in every simple bridgeless cubic graph there exist two perfect matchings such that the complement of their union is a bipartite graph. Here, we show that this conjecture can be equivalently stated using a variant of \mbox{Petersen-colourings}, we prove it for graphs having oddness at most four and we give a natural extension to bridgeless cubic multigraphs and to certain cubic graphs having bridges.
\end{abstract}

\end{frontmatter}   


\section{Introduction and terminology} 
Many interesting problems in graph theory are about the behaviour of perfect matchings in cubic graphs. One of the early classical results was made by Petersen \cite{Petersen} and states that every bridgeless cubic graph has at least one perfect matching. Some years ago, one of the most prominent conjectures in this area was completely solved by Esperet et al. in \cite{EsperetLovaszPlummer}: the conjecture, proposed by Lov\'{a}sz and Plummer in the 1970s, stated that the number of perfect matchings in a bridgeless cubic graph grows exponentially with its order (see \cite{LovaszPlummer}). However, many others are still open, such as Conjecture \ref{Conjecture BF} proposed independently by Berge and Fulkerson in the 1970s as well, and Conjecture \ref{Conjecture FR} by Fan and Raspaud (see \cite{BergeFulkerson} and \cite{FanRaspaud}, respectively). These two conjectures are related to the behaviour of the union and intersection of sets of perfect matchings, and properties of this kind are already largely studied: see, amongst others, \cite{AbreuEtAl,BonisoliCariolaro,JinSteffen,FractionalPM,KaiserRaspaud,LinWang,MacajovaSkovieraOddness2,MazzuoccoloCovering,MiaoWangZhang,Steffen,ZhuTangZhang}.
In this paper we prove that a seemingly stronger version of the \mbox{Fan-Raspaud} Conjecture is equivalent to the classical formulation (Theorem \ref{Theorem Main Equivalence FR}), and so to another interesting formulation proposed in \cite{MacajovaSkovieraOddCuts} (see also \cite{KMPRSS}). In the second part of the paper (Section \ref{Section S4} and Section \ref{Section New Results S4}), we study a weaker conjecture proposed by the first author in \cite{MazzuoccoloS4}: we show how we can state it in terms of a variant of Petersen-colourings (Proposition \ref{Proposition Equivalence}) and we prove it for cubic graphs of oddness four (Theorem \ref{TheoremOddness4}). Although all mentioned conjectures are about simple cubic graphs without bridges, we extend our study of the union of two perfect matchings to bridgeless cubic multigraphs and to particular cubic graphs having bridges (Section \ref{Section Multigraphs} and Section \ref{Section Graphs with bridges}).

Graphs considered in the sequel, unless otherwise stated, are simple connected bridgeless cubic graphs and so do not contain loops and parallel edges. Graphs that may contain parallel edges will be referred to as \emph{multigraphs}. For a graph $G$, let $V(G)$ and $E(G)$ be the set of vertices and the set of edges of $G$, respectively. A \textit{matching} of $G$ is a subset of $E(G)$ such that any two of its edges do not share a common vertex. For an integer $k\geq 0$, a \emph{k-factor} of $G$ is a spanning subgraph of $G$ (not necessarily connected) such that the degree of every vertex is $k$. The edge-set of a 1-factor is said to be a \emph{perfect matching}. The least number of odd cycles amongst all 2-factors of $G$, denoted by $\omega(G)$, is called the \emph{oddness} of $G$, and is clearly even for a cubic graph since $G$ has an even number of vertices. For $M \subseteq E(G)$, we denote the graph $G\setminus M$ by $\overline{M}$. In particular, when $M$ is a perfect matching of $G$, then $\overline{M}$ is a 2-factor of $G$. In this case, following the terminology used for instance in \cite{FiolMazSteffen}, if $\overline{M}$ has $\omega(G)$ odd cycles, then $M$ is said to be a \emph{minimal perfect matching}.

A \emph{cut} in $G$ is any set $X\subseteq E(G)$ such that $\overline{X}$ has more components than $G$, and no proper subset of $X$ has this property, i.e.  for any $X'\subset X$, $\overline{X'}$ does not have more components than $G$. The set of edges with precisely one end in $W\subseteq V(G)$ is denoted by $\partial_{G}W$, or just $\partial W$ when it is obvious to which graph we are referring. Moreover, a cut $X$ is said to be \emph{odd} if there exists a subset $W$ of $V(G)$ having odd cardinality such that $X=\partial W$.

We next define some standard operations on graphs that will be useful in the sequel. Let $G_{1}$ and $G_{2}$ be two bridgeless graphs (not necessarily cubic), and let $e_{1}$ and $e_{2}$ be two edges such that $e_{1}=u_{1}v_{1} \in E(G_{1})$ and $e_{2}=u_{2}v_{2} \in E(G_{2})$. A \emph{$2$-cut connection} on $u_{1}v_1$ and $u_{2}v_2$ is a graph operation that consists of constructing the new graph $$[G_{1}-e_{1}]\cup[G_{2}-e_{2}]\cup\{u_{1}u_{2}, v_{1}v_{2}\},$$ and denoted by $G_{1}(u_{1}v_{1})*G_{2}(u_{2}v_{2})$. It is clear that another possible graph obtained by a 2-cut connection on $e_{1}$ and $e_2$ is $G_{1}(u_{1}v_{1})*G_{2}(v_{2}u_{2})$. Clearly, the two graphs obtained are bridgeless, and, unless otherwise stated, if it is not important which of these two graphs is obtained, we use the notation $G_{1}(e_{1})*G_{2}(e_{2})$ and we say that it is a graph obtained by a $2$-cut connection on $e_{1}$ and $e_{2}$. 

Now, let $G_{1}$ and $G_{2}$ be two bridgeless cubic graphs, $v_{1}\in V(G_{1})$ and $v_{2}\in V(G_{2})$ such that the vertices adjacent to $v_{1}$ are $x_{1},y_{1}$ and $z_{1}$, and those adjacent to $v_{2}$ are $x_{2},y_{2}$ and $z_{2}$. A \emph{$3$-cut connection} (sometimes also known as the star product, see for instance \cite{StarProduct}) on $v_{1}$ and $v_{2}$ is a graph operation that consists of constructing the new graph $$[G_{1}-v_{1}]\cup[G_{2}-v_{2}]\cup\{x_{1}x_{2}, y_{1}y_{2}, z_{1}z_{2}\},$$ and denoted by $G_{1}(x_{1}y_{1}z_{1})*G_{2}(x_{2}y_{2}z_{2})$. The 3-edge-cut $\{x_{1}x_{2}, y_{1}y_{2}, z_{1}z_{2}\}$ is referred to as the \emph{principal 3-edge cut} (see for instance \cite{FouquetVanherpe}). As in the case of $2$-cut connections, other graphs can be obtained by a 3-cut connection on $v_{1}$ and $v_{2}$, and, unless otherwise stated, if it is not important how the adjacencies in the principal 3-edge cut look like, we use the notation $G_{1}(v_{1})*G_{2}(v_{2})$ and we say that it is a graph obtained by a $3$-cut connection on $v_{1}$ and $v_{2}$. It is clear that any resulting graph is also bridgeless and cubic.

\section{A list of relevant conjectures}\label{Section List Conjectures}
One of the aims of this paper is to study the behaviour of perfect matchings in cubic graphs, more specifically the union of two perfect matchings (see Section \ref{Section S4} and Section \ref{Section New Results S4}). We relate this to well-known conjectures stated here below, in particular: the \mbox{Berge-Fulkerson} Conjecture and the \mbox{Fan-Raspaud} Conjecture. 


\begin{conjecture}[\mbox{Berge-Fulkerson} \cite{BergeFulkerson}]\label{Conjecture BF}
Every bridgeless cubic graph $G$ admits six perfect matchings $M_{1}, \ldots, M_{6}$ such that any edge of $G$ belongs to exactly two of them.
\end{conjecture}
\begin{figure}[h]
      \centering
      \includegraphics[width=1\textwidth]{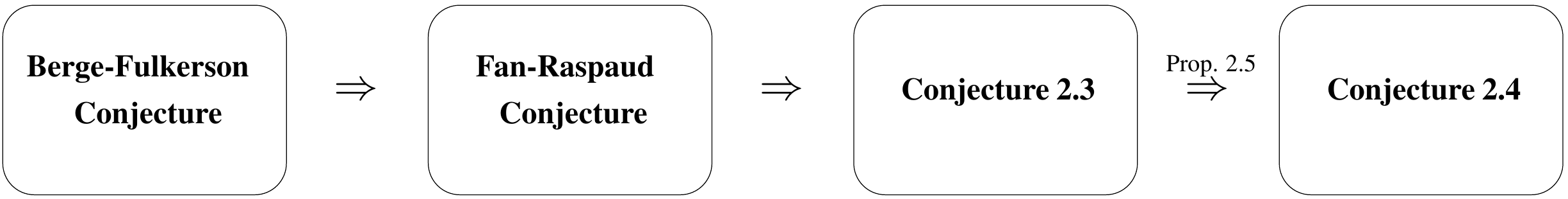}
	  \caption{Conjectures mentioned and how they are related.}
\end{figure}

We also state here other (possibly weaker) conjectures implied by the above conjecture.

\begin{conj}[\mbox{Fan-Raspaud} \cite{FanRaspaud}]\label{Conjecture FR}
Every bridgeless cubic graph admits three perfect matchings $M_{1}, M_{2},$ and $M_{3}$ such that $M_{1}\cap M_{2}\cap M_{3}=\emptyset$.
\end{conj} 
In the sequel we will refer to three perfect matchings satisfying Conjecture \ref{Conjecture FR} as an \emph{FR-triple}. We can see that Conjecture \ref{Conjecture FR} is immediately implied by the \mbox{Berge-Fulkerson} Conjecture, since we can take any three perfect matchings out of the six which satisfy Conjecture \ref{Conjecture BF}. A still weaker statement implied by the \mbox{Fan-Raspaud} Conjecture is the following:
\begin{conj}\cite{MacajovaSkovieraOddCuts}\label{ConjectureOddCut}
For each bridgeless cubic graph $G$, there exist two perfect matchings $M_{1}$ and $M_{2}$ such that $M_{1}\cap M_{2}$ contains no \mbox{odd-cut} of $G$. 
\end{conj}
We claim that any two perfect matchings out of the three in an \mbox{FR-triple} have no \mbox{odd-cut} in their intersection, in other words that Conjecture \ref{Conjecture FR} implies Conjecture \ref{ConjectureOddCut}. For, suppose not. Then, without loss of generality, suppose that $M_{2}\cap M_{3}$ contains an \mbox{odd-cut} $X$. Hence, since every perfect matching has to intersect an \mbox{odd-cut} at least once, \break $\vert M_{1}\cap (M_{2}\cap M_{3})\vert \geq \vert M_{1} \cap X\vert\geq 1$, a contradiction, since we assumed that $M_{1}\cap M_{2}\cap M_{3}$ is empty. In relation to the above, the first author proposed the following conjecture:
\begin{conj}[$S_4$-Conjecture \cite{MazzuoccoloS4}]\label{Conjecture S4 Matchings}
For any bridgeless cubic graph $G$, there exist two perfect matchings such that the deletion of their union leaves a bipartite subgraph of $G$. 
\end{conj}
For reasons which shall be obvious in Section \ref{Section S4} we let such a pair of perfect matchings be called an \textit{\mbox{$S_{4}$-pair}} of $G$ and shall refer to Conjecture \ref{Conjecture S4 Matchings} as the \mbox{$S_{4}$-Conjecture}. We will first proceed by showing that this conjecture is implied by Conjecture \ref{ConjectureOddCut}, and so, by what we have said so far, is a consequence of the \mbox{Berge-Fulkerson} Conjecture. In particular, we can see the \mbox{$S_{4}$-Conjecture} as Conjecture \ref{ConjectureOddCut} restricted to \mbox{odd-cuts} $\partial V(C)$, where $C$ is an odd cycle of $G$.

\begin{prop}
Conjecture \ref{ConjectureOddCut} implies the \mbox{$S_{4}$-Conjecture}.
\end{prop}
\begin{proof}
Let $M_{1}$ and $M_{2}$ be two perfect matchings such that their intersection does not contain any \mbox{odd-cut}. Consider $\overline{M_{1}\cup M_{2}}$, and suppose that it contains an odd cycle $C$. Then, all the edges of $\partial V(C)$ belong to $M_{1}\cap M_{2}$. If $\overline{\partial V(C)}$ has exactly two components, then $\partial V(C)$ is an \mbox{odd-cut} belonging to $M_{1}\cap M_{2}$, a contradiction. Therefore, $\overline{\partial V(C)}$ must have more than two components, say $k$, denoted by $C_{1}, C_{2}, \ldots, C_{k}$, where the first component $C_{1}$ is the cycle $C$. Let $[C_{1}, C_{j}]$ denote the set of edges between $C_{1}$ and $C_{j}$, for $j\in \{2,\ldots,k\}$. Since $\sum_{j=2}^{k}\vert [C_{1}, C_{j}]\vert=\vert \partial V(C) \vert\equiv 1\mod 2$, there exists $j'\in \{2,\ldots,k\}$, such that $\vert[C_{1}, C_{j'}]\vert\equiv 1 \mod 2$. However, $[C_{1}, C_{j'}]$ is an \mbox{odd-cut} which belongs to $M_{1}\cap M_{2}$, a contradiction. 
\end{proof}
\section{Statements equivalent to the \mbox{Fan-Raspaud} Conjecture}\label{Section FR}
Let $M_{1}, \ldots, M_{t}$ be a list of perfect matchings of $G$, and let $a\in E(G)$. We denote the number of times $a$ occurs in this list by $\nu_{G}[a:M_{1}, \ldots, M_{t}]$. When it is obvious which list of perfect matchings or which graph we are referring to, we will denote this as $\nu(a)$ and refer to it as the \emph{frequency of $a$}. We will sometimes need to refer to the frequency of an ordered list of edges, say $(a,b,c)$, and we will do this by saying that the frequency of $(a,b,c)$ is $(i,j,k)$, for some integers $i,j$ and $k$. Mkrtchyan et al. \cite{Vahan} showed that the \mbox{Fan-Raspaud} Conjecture, i.e. Conjecture \ref{Conjecture FR}, is equivalent to the following:
\begin{conj}\cite{Vahan}\label{Conjecture FR*}
For each bridgeless cubic graph $G$, any edge $a\in E(G)$ and any $i\in\{0,1,2\}$, there exist three perfect matchings $M_{1}, M_{2},$ and $M_{3}$ such that $M_{1}\cap M_{2}\cap M_{3}=\emptyset$ and $\nu_{G}[a:M_{1}, M_{2}, M_{3}]=i$.
\end{conj}
In other words they show that if a graph has an \mbox{FR-triple} then, for every $i$ in $\{0,1,2\}$, there exists an \mbox{FR-triple} in which the frequency of a pre-chosen edge is exactly $i$. In the same paper, Mkrtchyan et al. state the following seemingly stronger version of the \mbox{Fan-Raspaud} Conjecture:
\begin{conj}\cite{Vahan}\label{Conjecture FR**}
Let $G$ be a bridgeless cubic graph, $w$ a vertex of $G$ and $i,j$ and $k$ three integers in $\{0,1,2\}$ such that $i+j+k=3$. Then, $G$ has an \mbox{FR-triple} in which the edges incident to $w$ in a given order have frequencies $(i,j,k)$.
\end{conj}
This means that we can prescribe the frequencies to the three edges incident to a given vertex. At the end of \cite{Vahan}, the authors remark that it would be interesting to show that Conjecture \ref{Conjecture FR**} is equivalent to the \mbox{Fan-Raspaud} Conjecture. We prove here that this is actually the case.
\begin{theorem}\label{Theorem Main Equivalence FR}
Conjecture \ref{Conjecture FR**} is equivalent to the \mbox{Fan-Raspaud} Conjecture.
\end{theorem}
\begin{proof}
Since the \mbox{Fan-Raspaud} Conjecture is equivalent to Conjecture \ref{Conjecture FR*}, it suffices to show the equivalence of Conjectures \ref{Conjecture FR*} and \ref{Conjecture FR**}. The latter clearly implies the former, so assume Conjecture \ref{Conjecture FR*} is true and let $a,b$ and $c$ be the edges incident to $w$ such that the frequencies $(i,j,k)$ are to be assigned to $(a,b,c)$. It is sufficient to show that there exist two \mbox{FR-triples} in which the frequencies of $(a,b,c)$ are $(2,1,0)$ in one \mbox{FR-triple} (Case 1 below) and $(1,1,1)$ in the other \mbox{FR-triple} (Case 2 below).
\begin{figure}[h]
      \centering
      \includegraphics[width=0.45\textwidth]{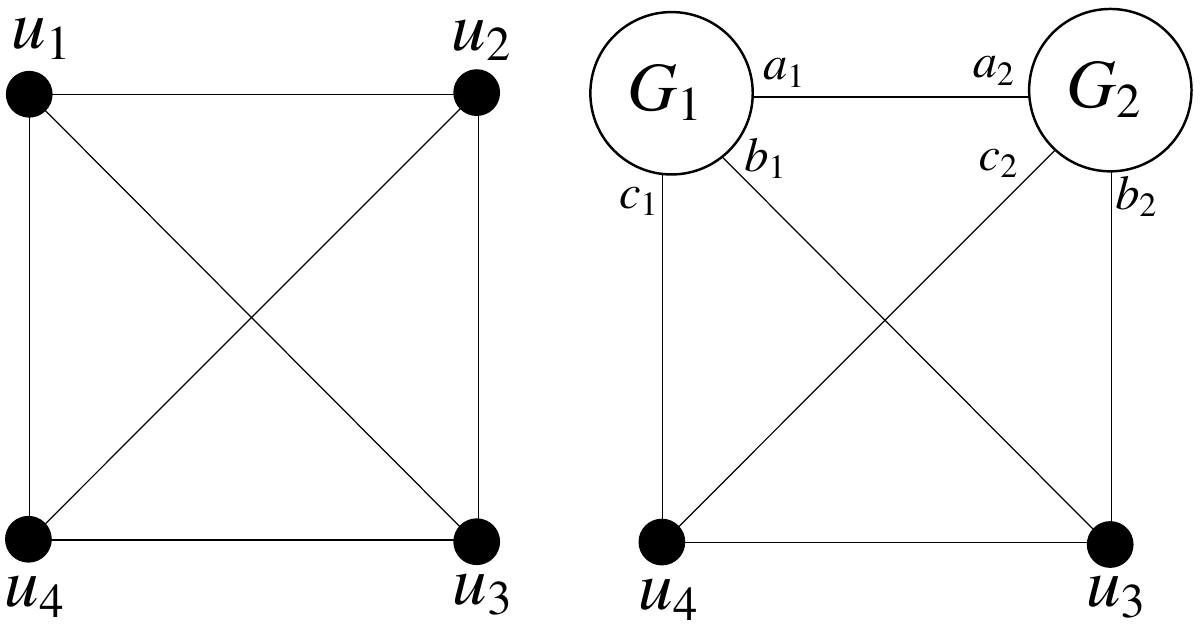}
	  \caption{The graphs $K_4$ and $K^*_4$ in Case 1 of the proof of Theorem \ref{Theorem Main Equivalence FR}.}
      \label{Figure FR210}
\end{figure}

\noindent \textbf{Case 1.} Let $u_1,u_2,u_3$ and $u_4$ be the vertices of the complete graph $K_{4}$ as in Figure \ref{Figure FR210}. Consider two copies of $G$, and let the vertex $w$ in the $i^{th}$ copy of $G$ be denoted by $w_{i}$, for each $i\in\{1,2\}$. We apply a $3$-cut connection between $u_i$ and $w_i$, for each $i\in\{1,2\}$. With reference to this resulting graph, denoted by $K^*_4$, we refer to the copy of the graph $G-w$ at $u_1$ as $G_1$, and to the corresponding edges $a,b$ and $c$ as $a_1, b_1$ and $c_1$, respectively. The graph $G_2$ and the edges $a_2, b_2$ and $c_2$ are defined in a similar way, and the $3$-cut connection is done in such a way that $b_{1}$ and $b_{2}$ are adjacent, and also $c_{1}$ and $c_{2}$, as Figure \ref{Figure FR210} shows. Note also that $a_1$ and $a_2$ coincide in $K^*_4$. By our assumption, there exists an \mbox{FR-triple} $M_{1}, M_{2}$ and $M_{3}$ of $K^*_4$ in which the edge $u_3u_4$ has frequency $2$. Without loss of generality, let $u_3u_4 \in M_{1}\cap M_{2}$. Then, $a_1$ (and so $a_2$) must belong to $M_{1}\cap M_{2}$. Clearly, $a_1$ (and so $a_2$) cannot belong to $M_{3}$, and so the principal 3-edge-cuts with respect to $G_1$ and $G_2$ do not belong to $M_{3}$. If $b_1\in M_{3}$, then we are done, as then $M_{1}, M_{2}$ and $M_{3}$ restricted to $G_1$, together with $a$ and $b$ having the same frequencies as $a_1$ and $b_1$, induce an \mbox{FR-triple} of $G$ such that the frequencies of $(a,b,c)$ are $(2,1,0)$. So suppose $c_1\in M_{3}$. Then, $b_2\in M_{3}$, and so by a similar argument applied to $G_2$ and the corresponding edges, $M_{1}, M_{2}$ and $M_{3}$ induce an \mbox{FR-triple} in $G$ such that the frequencies of $(a,b,c)$ are $(2,1,0)$.
\begin{figure}[h]
      \centering
      \includegraphics[width=0.8\textwidth]{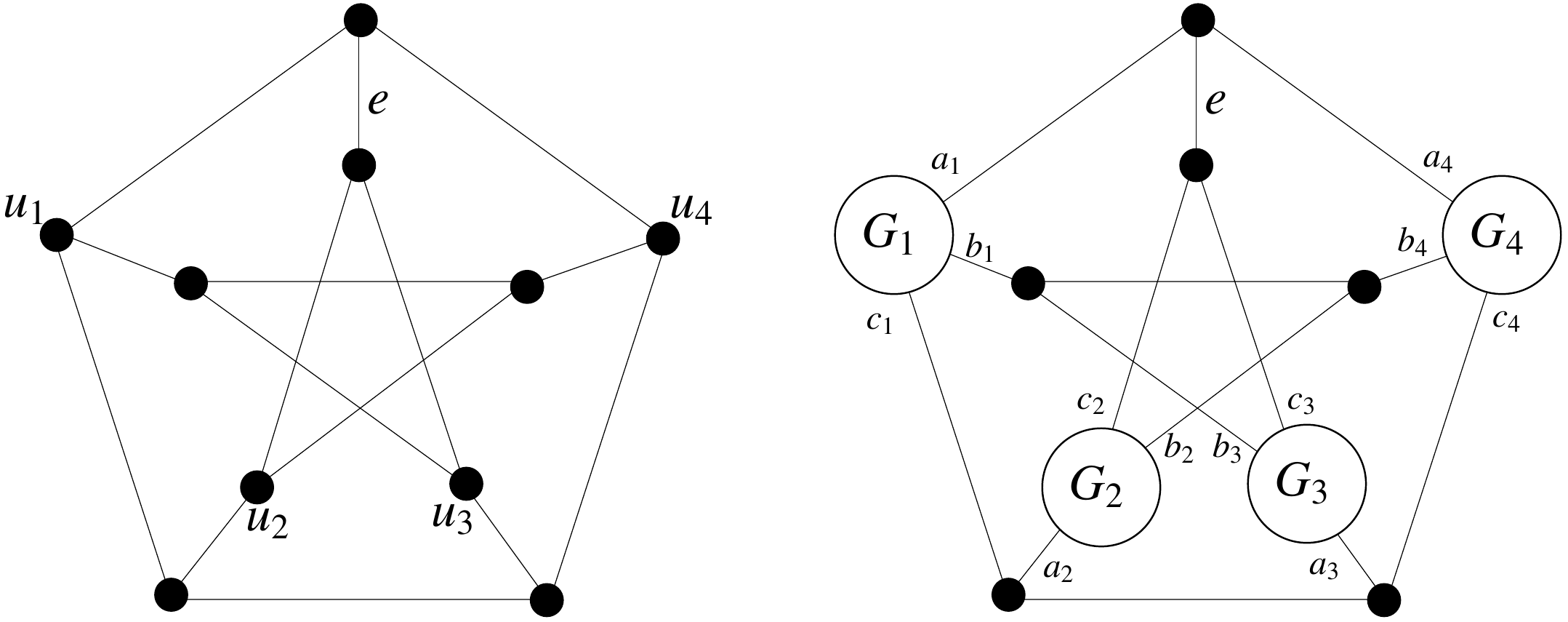}
	  \caption{The graphs $P$ and $P^*$ in Case 2 of the proof of Theorem \ref{Theorem Main Equivalence FR}.}
      \label{Figure FR111}
\end{figure}

\noindent \textbf{Case 2.} Let $P$ be the Petersen graph and $\{u_1,u_2,u_3,u_4\}$ be a maximum independent set of vertices in $P$ as in Figure \ref{Figure FR111}. Consider four copies of $G$. Let the vertex $w$ in the $i^{th}$ copy of $G$ be denoted by $w_{i}$, for each $i\in\{1,\ldots,4\}$. Let $P^{*}$ be the graph obtained by applying a $3$-cut connection between each $u_i$ and $w_{i}$, as shown in Figure \ref{Figure FR111}. Similar to Case 1 we refer to the copy of $G-w$ at $u_i$ as $G_i$ and to the corresponding edges $a,b$ and $c$ as $a_i, b_i$ and $c_i$, respectively. Since we are assuming that Conjecture \ref{Conjecture FR*} is true, we can consider an \mbox{FR-triple} $M_{1}, M_{2}$ and $M_{3}$ of $P^*$ in which the edge $e$ incident to both $a_1$ and $a_4$ has frequency $2$. Without loss of generality, let the two perfect matchings containing $e$ be $M_{1}$ and $M_{2}$. The edges $a_1, c_2, c_3$ and $a_4$ are not contained in $M_{1}$ and neither $M_{2}$, since they are all incident to $e$, and so no principal $3$-edge-cut leaving $G_i$ belongs to $M_1$ or $M_2$. Then, $M_1$ and $M_2$ induce perfect matchings of $P$ (clearly distinct), and since there are exactly two perfect matchings of $P$ containing $e$, we can assume that $M_{1}$ contains $\{e,b_1,a_2,a_3,b_4\}$, and $M_{2}$ contains $\{e,c_1,b_2,b_3,c_4\}$.

If the third perfect matching $M_3$ induces a perfect matching of the Petersen graph then the induced perfect matching cannot be one of the perfect matchings induced by $M_1$ and $M_2$ in $P$. Hence, since every two distinct perfect matchings of $P$ intersect in exactly one edge of $P$, there exists $i\in\{1,2,3,4\}$ such that the frequencies of $(a_i,b_i,c_i)$ are $(1,1,1)$ and so, $M_{1}, M_{2}$ and $M_{3}$ restricted to $G_{i}$, together with $a,b$ and $c$ having the same frequencies as $a_i,b_i$ and $c_i$, induce an \mbox{FR-triple} in $G$ with the needed property.

Therefore, suppose $M_{3}$ contains the principal 3-edge-cut of one of the $G_i$s, say $G_1$ by symmetry of $P^{*}$. Thus, $a_1,b_1$ and $c_1$ belong to $M_3$. 
The perfect matching $M_{3}$ can intersect the principal 3-edge-cut at $G_{2}$ either in $b_{2}$ or $c_{2}$ (not both). If $c_{2}\in M_{3}$ we are done by the same reasoning above applied to $G_{2}$ and the corresponding edges. So suppose $b_{2}\in M_{2}\cap M_{3}$. Then, $c_{4}\in M_{3}$, and $M_{3}$ can only intersect the principal 3-edge-cut at $G_{3}$ in $c_{3}$, implying that the frequencies of $(a_3,b_3,c_3)$ are $(1,1,1)$ in $P^*$ and that $M_{1}, M_{2}$ and $M_{3}$ restricted to $G_{3}$, together with $a,b$ and $c$ having the same frequencies as $a_3,b_3$ and $c_3$, induce an \mbox{FR-triple} in $G$ with the needed property.
\end{proof}

In \cite{Vahan} it is also shown that a minimal counterexample to Conjecture \ref{Conjecture FR**} is cyclically 4-edge-connected. It remains unknown whether a smallest counterexample to the original formulation of the \mbox{Fan-Raspaud} Conjecture has the same property. Indeed, we only prove that the two assertions are equivalent, but we cannot say whether a possible counterexample to Conjecture \ref{Conjecture FR**} is itself a counterexample to the original formulation. 

\section{Statements equivalent to the \mbox{\boldmath{$S_{4}$}-Conjecture}}\label{Section S4}
All conjectures presented in Section \ref{Section List Conjectures} are implied by a conjecture made by Jaeger in the late 1980s. In order to state it we need the following definitions. Let $G$ and $H$ be two graphs. An \emph{$H$-colouring} of $G$ is a proper edge-colouring $f$ of $G$ with edges of $H$, such that for each vertex $u\in V(G)$, there exists a vertex $v \in V(H)$ with $f(\partial_{G}\{u\})\subseteq\partial_{H}\{v\}$. If $G$ admits an $H$-colouring, then we will write $H\prec G$. In this paper we consider \mbox{$S_{4}$-colourings} of bridgeless cubic graphs, where $S_{4}$ is the multigraph shown in Figure \ref{FigureS4}. 
\begin{figure}[h]
      \centering
      \includegraphics[width=0.35\textwidth]{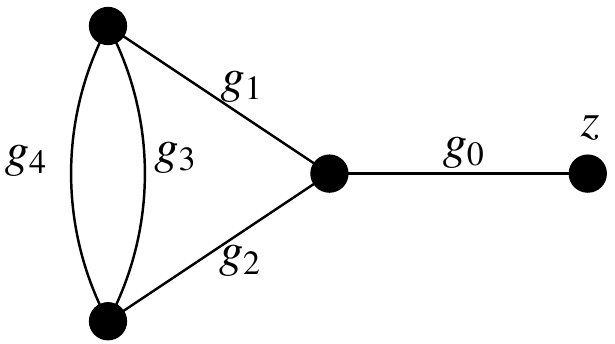}
      \caption{The multigraph $S_{4}$.}
      \label{FigureS4}
\end{figure}

The importance of $H$-colourings is mainly due to Jaeger's Conjecture \cite{Jaeger} which states that each bridgeless cubic graph $G$ admits a $P$-colouring (where $P$ is again the Petersen graph). For recent results on $P$-colourings, known as Petersen-colourings, see for instance \cite{HakMkr,VahanPetersen,Samal}. The following proposition shows why we choose to refer to a pair of perfect matchings whose deletion leaves a bipartite subgraph as an \mbox{$S_{4}$-pair}.
\begin{prop} \label{Proposition Equivalence}
Let $G$ be a bridgeless cubic graph, then $S_{4}\prec G$ if and only if $G$ has an \mbox{$S_{4}$-pair}.
\end{prop}
\begin{proof}
Along the entire proof we denote the edges of $S_4$ by using the same labelling as in Figure \ref{FigureS4}.
Let $M_{1}$ and $M_{2}$ be an \mbox{$S_{4}$-pair} of $G$. The graph induced by $M_{1}\cup M_{2}$, denoted by $G[M_{1}\cup M_{2}]$, is made up of even cycles and isolated edges, whilst the bipartite graph $\overline{M_{1}\cup M_{2}}$ is made up of even cycles and paths.  
We obtain an \mbox{$S_{4}$-colouring} of $G$ as follows:
\begin{itemize}
\item the isolated edges in $M_{1}\cup M_{2}$ are given colour $g_{0}$,
\item the edges of the even cycles in $M_{1}\cup M_{2}$ are properly edge-coloured with $g_{3}$ and $g_{4}$, and
\item the edges of the paths and even cycles in $\overline{M_{1}\cup M_{2}}$ are properly edge-coloured with $g_{1}$ and $g_{2}$.
\end{itemize}
One can clearly see that this gives an \mbox{$S_{4}$-colouring} of $G$. Conversely, assume that $S_{4}\prec G$. We are required to show that there exists an \mbox{$S_{4}$-pair} of $G$. Let $M_{1}$ be the set of edges of $G$ coloured $g_{3}$ and $g_{0}$, and let $M_{2}$ be the set of edges of $G$ coloured $g_{4}$ and $g_{0}$. If $e$ and $f$ are edges of $G$ coloured $g_{3}$ (or $g_{4}$) and $g_{0}$, respectively, then $e$ and $f$ cannot be adjacent, otherwise we contradict the \mbox{$S_{4}$-colouring} of $G$. Thus, $M_{1}$ and $M_{2}$ are matchings. Next, we show that they are in fact perfect matchings. This follows since for every vertex $v$ of $G$, $f(\partial_{G}\{v\})$ is equal to $\{g_{1}, g_{3}, g_{4}\}$, or $\{g_{2}, g_{3}, g_{4}\}$, or $\{g_{0}, g_{1}, g_{2}\}$. Thus, $\overline{M_{1}\cup M_{2}}$ is the graph induced by the edges coloured $g_{1}$ and $g_{2}$, which clearly cannot induce an odd cycle.
\end{proof}
Hence, by the previous proof, Conjecture \ref{Conjecture S4 Matchings} can be stated in terms of \mbox{$S_{4}$-colourings}, which clearly shows why we choose to refer to it as the \mbox{$S_{4}$-Conjecture}. 
In analogy to what we did for \mbox{FR-triples}, here we prove that for \mbox{$S_{4}$-pairs} we can prescribe the frequency of an edge and the frequencies of the edges leaving a vertex (the proof of the latter also implies that we can prescribe the frequencies of the edges of each $3$-cut). 
Consider the following conjecture, analogous to Conjecture \ref{Conjecture FR*}:
\begin{conj}\label{Conjecture S4*}
For any bridgeless cubic graph $G$, any edge $a \in E(G)$ and any $i\in\{0,1,2\}$, there exists an \mbox{$S_{4}$-pair}, say $M_{1}$ and $M_{2}$, such that $\nu_{G}[a:M_{1},M_{2}]=i$.
\end{conj}
In Theorem \ref{Theorem Equivalence S4 and S4*} we show that the latter conjecture is actually equivalent to the \break \mbox{$S_{4}$-Conjecture}. The proof given in \cite{Vahan} to show the equivalence of the \mbox{Fan-Raspaud} Conjecture and Conjecture \ref{Conjecture FR*} is very similar to the proof we give here for the analogous case for the \mbox{$S_{4}$-Conjecture}, however, we need a slightly more complicated tool in our context.
\begin{theorem}\label{Theorem Equivalence S4 and S4*}
Conjecture \ref{Conjecture S4*} is equivalent to the \mbox{$S_{4}$-Conjecture}.
\end{theorem}
\begin{figure}[h]
      \centering
      \includegraphics[width=0.75\textwidth]{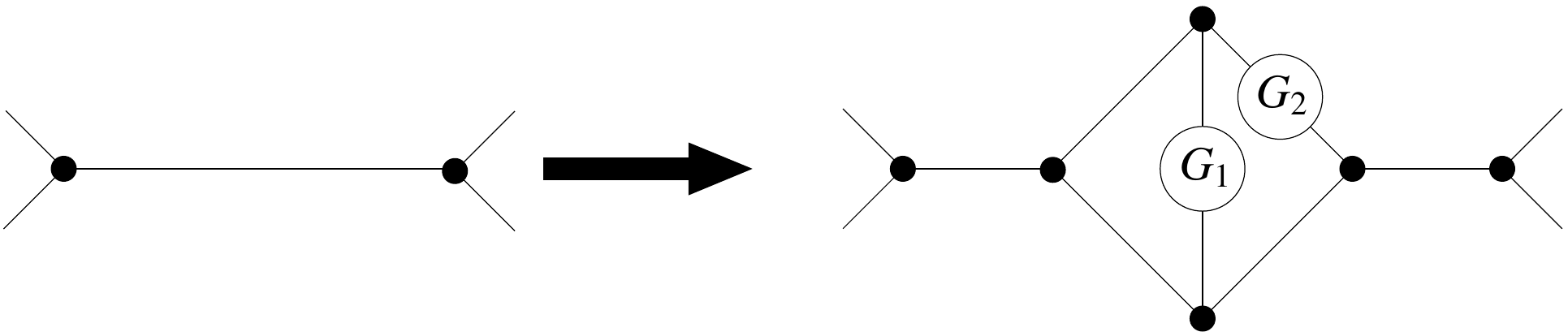}
      \caption{An edge in $P$ transformed into the corresponding structure in $H$.}
      \label{Figure GraphH}
\end{figure}
\begin{proof}
Clearly, Conjecture \ref{Conjecture S4*} implies the \mbox{$S_{4}$-Conjecture} so it suffices to show the converse. Assume the \mbox{$S_{4}$-Conjecture} to be true and let $f_{1}$, $f_{2}$, $f_{3}$ be three consecutive edges of $K_{4}$ inducing a path.
Consider two copies of $G$. Let the edge $a$ in the $i^{th}$ copy of $G$ be denoted by $a_{i}$, for each $i\in\{1,2\}$. Let $K_4'$ be the graph obtained by applying a $2$-cut connection between $f_i$ and $a_{i}$ for each $i\in\{1,2\}$. We refer to the copy of the graph $G-a$ on $f_i$ as $G_i$.

Let $\{e_{1}, \ldots, e_{15}\}$ be the edges of the Petersen graph and let $T_{1}, \ldots, T_{15}$ be 15 copies of $K_{4}'$. For every $i\in\{1,\ldots,15\}$, apply a 2-cut connection on $e_{i}$ and the edge $f_{3}$ of $T_{i}$. Consequently, every edge $e_{i}$ of the Petersen graph is transformed into the structure $E_{i}$ as in Figure \ref{Figure GraphH}, and we refer to $G_{1}$ and $G_{2}$ on $E_{i}$ as $G_{1}^{i}$ and $G_{2}^{i}$, respectively. Let $H$ be the resulting graph. By our assumption, there exists an \mbox{$S_{4}$-pair} of $H$, say $M_{1}$ and $M_{2}$, which induces a pair of two distinct perfect matchings in $P$, say $N_{1}$ and $N_{2}$, respectively. There exists an edge of $P$, say $e_{j}$, for some $j\in\{1,\ldots,15\}$, such that $\nu_{P}[e_{j}:N_{1},N_{2}]=1$, since every two distinct perfect matchings of $P$ have exactly one edge of $P$ in common. Hence, the restriction of $M_{1}$ and $M_{2}$ to the edge set of $G_{1}^{j}$, together with the edge $a$ having the same frequency as $e_{j}$, gives rise to an \mbox{$S_{4}$-pair} of $G$ in which the frequency of $a$ is $1$.

Moreover, there exists an edge of $P$, say $e_{k}$, for some $k\in\{1,\ldots,15\}$, such that $\nu_{P}[e_{k}:N_{1}, N_{2}]=2$. Restricting $M_{1}$ and $M_{2}$ to the edge set of $G_{1}^{k}$, together with the edge $a$ having the same frequency as $e_{k}$, gives rise to an \mbox{$S_{4}$-pair} of $G$, in which the frequency of $a$ is $2$. Also, the restriction of $M_{1}$ and $M_{2}$ to the edge set of $G_{2}^{k}$ gives rise to an \mbox{$S_{4}$-pair} of $G$ ($G_{2}^{k}$ together with $a$), in which the frequency of $a$ is $0$, because if not, then there exists an odd cycle in $G$, say of length $\alpha$, passing through $a$ and having all its edges with frequency $0$. However, this would mean that there is an odd cycle of length $\alpha+4$ on $E_{k}$ in $\overline{M_{1}\cup M_{2}}$ (in $H$), a contradiction.
\end{proof}
As in Section \ref{Section FR}, we state an analogous conjecture to Conjecture \ref{Conjecture FR**}, but for \mbox{$S_{4}$-pairs}:
\begin{conj}\label{ConjectureS4**}
Let $G$ be a bridgeless cubic graph, $w$ a vertex of $G$ and $i,j$ and $k$ three integers in $\{0,1,2\}$ such that $i+j+k=2$. Then, $G$ has an \mbox{$S_{4}$-pair} in which the edges incident to $w$ in a given order have frequencies $(i,j,k)$.
\end{conj}
The following theorem shows that this conjecture is actually equivalent to Conjecture \ref{Conjecture S4*}, and so to the \mbox{$S_{4}$-Conjecture} by Theorem \ref{Theorem Equivalence S4 and S4*}.
\begin{theorem}\label{Theorem Equivalence S4 S4**}
Conjecture \ref{ConjectureS4**} is equivalent to the \mbox{$S_{4}$-Conjecture}.
\end{theorem}
\begin{proof}
Since the \mbox{$S_{4}$-Conjecture} is equivalent to Conjecture \ref{Conjecture S4*}, it suffices to show the equivalence of Conjectures \ref{Conjecture S4*} and \ref{ConjectureS4**}. Clearly, Conjecture \ref{ConjectureS4**} implies Conjecture \ref{Conjecture S4*} and so we only need to show the converse. Let $a,b$ and $c$ be the edges incident to $w$ such that the frequencies $(i,j,k)$ are to be assigned to $(a,b,c)$. We only need to prove the case when $(i,j,k)$ is equal to $(1,1,0)$, as all other cases follow from Conjecture \ref{Conjecture S4*}. 
\begin{figure}[h]
      \centering
      \includegraphics[width=0.4\textwidth]{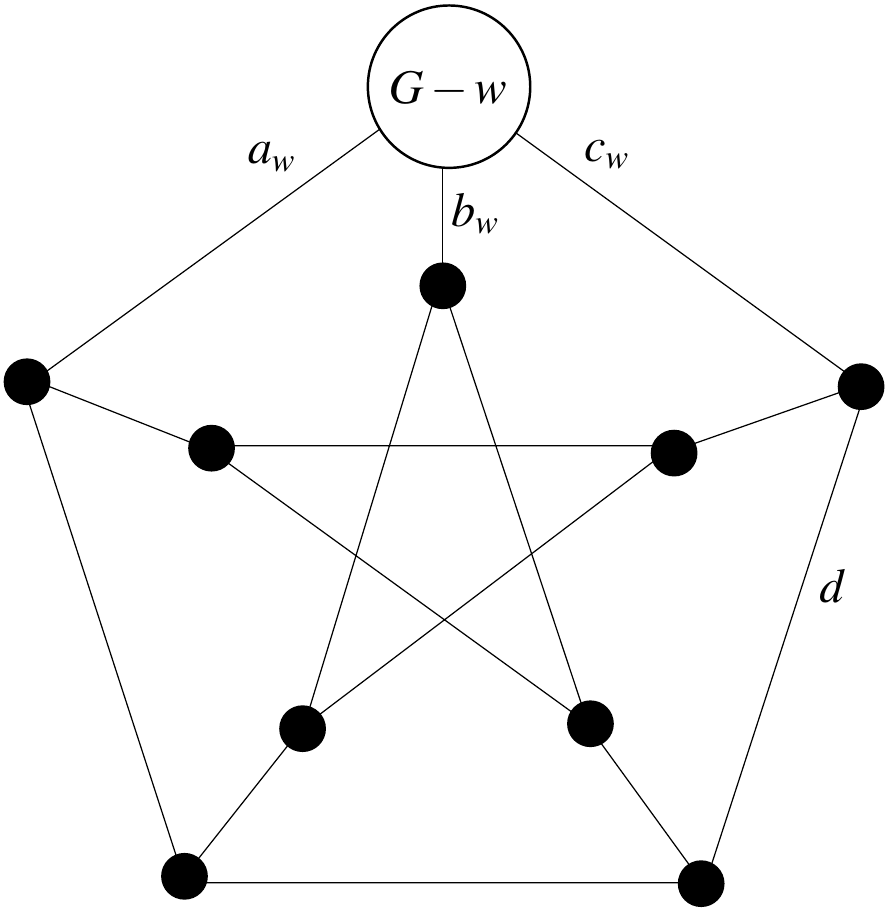}
      \caption{The graph $G(w)*P(v)$ from Theorem \ref{Theorem Equivalence S4 S4**}.}
      \label{Figure S4110}
\end{figure}

Consider the graph $G(w)*P(v)$, where $P$ is the Petersen graph and $v$ is any vertex of $P$. We refer to the edges corresponding to $a,b$ and $c$ in $G(w)*P(v)$, as $a_{w},b_{w}$ and $c_{w}$. Let $d$ be an edge originally belonging to $P$ and adjacent to $c_{w}$ in $G(w)*P(v)$. Since we are assuming Conjecture \ref{Conjecture S4*} to be true, there exists an \mbox{$S_{4}$-pair} in $G(w)*P(v)$ in which $d$ has frequency $2$. If the frequencies of $(a_{w},b_{w},c_{w})$ are $(1,1,0)$, then we are done, because the \mbox{$S_{4}$-pair} for $G(w)*P(v)$ restricted to the edges in $G-w$, together with $a$ and $b$ having the same frequencies as $a_{w}$ and $b_{w}$, give an \mbox{$S_{4}$-pair} for $G$ with the desired property. We claim that this must be the case. For, suppose not. Then, without loss of generality, the frequencies of $(a_{w},b_{w},c_{w})$ are $(2,0,0)$. This implies that all the edges of $G(w)*P(v)$ originally in $P$ have either frequency 0 or 2, since the two perfect matchings in the \mbox{$S_{4}$-pair} induce the same perfect matching in $P$. However, this implies that $P$ has a perfect matching whose complement is bipartite, a contradiction since $P$ is not $3$-edge-colourable.
\end{proof}

As in \cite{Vahan}, a minimal counterexample to Conjecture \ref{ConjectureS4**} (but not necessarily to the \mbox{$S_{4}$-Conjecture}) is cyclically 4-edge-connected. We omit the proof of this result as it is very similar to the proof of Theorem 2 in \cite{Vahan}.
\section{Further Results on the \mbox{\boldmath{$S_{4}$}-Conjecture}}\label{Section New Results S4}
Little progress has been made on the \mbox{Fan-Raspaud} Conjecture so far. Bridgeless cubic graphs which trivially satisfy this conjecture are those which can be edge-covered by four perfect matchings. In this case, every three perfect matchings from a cover of this type form an \mbox{FR-triple} since every edge has frequency one or two with respect to this cover. Therefore, a possible counterexample to the \mbox{Fan-Raspaud} Conjecture should be searched for in the class of bridgeless cubic graphs whose edge-set cannot be covered by four perfect matchings, see for instance \cite{EsperetMazzuoccolo}. In 2009, M\'{a}\v{c}ajov\'{a} and \v{S}koviera \cite{MacajovaSkovieraOddness2} shed some light on the \mbox{Fan-Raspaud} Conjecture by proving it for bridgeless cubic graphs having oddness two. One of the aims of this paper is to show that even if the \mbox{$S_{4}$-Conjecture} is still open, some results are easier to extend than the corresponding ones for the \mbox{Fan-Raspaud} Conjecture. Clearly, the result by M\'{a}\v{c}ajov\'{a} and \v{S}koviera in \cite{MacajovaSkovieraOddness2} implies the following result:

\begin{theorem}\label{Theorem S4 Oddness 2}
Let $G$ be a bridgeless cubic graph of oddness two. Then, $G$ has an \mbox{$S_{4}$-pair}. 
\end{theorem}
We first give a proof of Theorem \ref{Theorem S4 Oddness 2} in the same spirit of that used in \cite{MacajovaSkovieraOddness2}, however much shorter since we are proving a weaker result.\\

\noindent \emph{Proof 1 of Theorem \ref{Theorem S4 Oddness 2}.} Let $M_{1}$ be a minimal perfect matching of $G$, and let $C_{1}$ and $C_{2}$ be the two odd cycles in $\overline{M_{1}}$. Colour the even cycles in $\overline{M_{1}}$ using two colours, say $1$ and $2$. For each $i\in\{1,2\}$, let $E_{i}$ be the set of edges belonging to the even cycles in $\overline{M_{1}}$ and having colour $i$. In $G$, there must exist a path $Q$ whose edges alternate in $M_{1}$ and $E_{1}$ and whose end-vertices belong to $C_{1}$ and $C_{2}$, respectively, since $C_1$ and $C_2$ are odd cycles. Note that since the edges of $C_{1}$ and $C_{2}$ are not edges in $M_{1} \cup E_{1}$, every other vertex on $Q$ which is not an end-vertex does not belong to $C_{1}$ and $C_{2}$. 

For each $i\in\{1,2\}$, let $v_{i}$ be the end-vertex of $Q$ belonging to $C_{i}$, and let $M_{C_{i}}$ be the unique perfect matching of $C_{i}-v_{i}$. Let $M_{2}:= (M_{1}\cap Q)\cup (E_{1}\setminus Q)\cup M_{C_{1}} \cup M_{C_{2}}$. Clearly, $M_{2}$ is a perfect matching of $G$ which intersects $C_{1}$ and $C_{2}$, and so $\overline{M_{1} \cup M_{2}}$ is bipartite. \hfill $\square$\\

We now give a second alternative proof of the same theorem using fractional perfect matchings, which we will show to be easier to use for graphs having larger oddness. Let $w$ be a vector in $\mathbb{R}^{\vert E(G)\vert}$. The entry of $w$ corresponding to $e\in E(G)$ is denoted by $w(e)$, and for $A\subseteq E(G)$, we let the weight of $A$, denoted by $w(A)$, to be equal to $\sum_{e\in A}w(e)$. The vector $w$ is said to be a \emph{fractional perfect matching} of $G$ if:
\begin{enumerate}
\item $w(e)\in [0,1]$ for each $e \in E(G)$,
\item $w(\partial\{v\})=1$ for each $v \in V(G)$, and
\item $w(\partial W)\geq 1$ for each $W \subseteq V(G)$ of odd cardinality.
\end{enumerate}
The following lemma is presented in \cite{FractionalPM} and it is a consequence of Edmonds' characterisation of perfect matching polytopes in \cite{Edmonds}.
\begin{lemma}\label{LemmaFractionalPM}
If $w$ is a fractional perfect matching in a graph $G$, and $c\in \mathbb{R}^{\vert E(G)\vert}$, then $G$ has a perfect matching $N$ such that $$c\cdot\chi^{N}\geq c\cdot w,$$ where $\cdot$ denotes the inner product. Moreover, there exists a perfect matching satisfying the above inequality and which contains exactly one edge of each \mbox{odd-cut} $X$ with $w(X)=1$. 
\end{lemma}
\begin{rem}\label{Remark3cuts}
If we let $w(e)=\sfrac{1}{3}$ for all $e\in E(G)$, for some graph $G$, then we know that $w$ is a fractional perfect matching of $G$. Also, since the weight of every $3$-cut is one, then by Lemma \ref{LemmaFractionalPM} there exists a perfect matching of $G$ containing exactly one edge of each 3-cut of $G$. 
\end{rem}

\noindent \emph{Proof 2 of Theorem \ref{Theorem S4 Oddness 2}.} Let $M_{1}$ be a minimal perfect matching of $G$, and let $C_{1}$ and $C_{2}$ be the two odd cycles in $\overline{M_{1}}$. For each $i\in\{1,2\}$, let $e_{1}^{i}$ and $e_{2}^{i}$ be two adjacent edges belonging to $C_{i}$. We define the vector $c\in \mathbb{R}^{\vert E(G)\vert}$ such that
$$c(e)= \left\{
\begin{array}{cl}
1 & $if $ e\in \cup_{i=1}^{2}\{e_{1}^{i}, e_{2}^{i}\},\\
0 & $otherwise.$
\end{array}\right.$$
By Remark \ref{Remark3cuts}, we also know that if we let $w(e)=\sfrac{1}{3}$ for all $e\in E(G)$, then $w$ is a fractional perfect matching of $G$. Hence, by Lemma \ref{LemmaFractionalPM}, there exists a perfect matching $M_{2}$ such that $c\cdot\chi^{M_{2}}\geq c\cdot w,$
which implies that $$\vert\cup_{i=1}^{2}\{e_{1}^{i}, e_{2}^{i}\} \cap M_{2} \vert\geq \sfrac{1}{3}\times 2 \times 2 = \sfrac{4}{3}>1.$$ 
Therefore, for each $i\in\{1,2\}$, there exists exactly one $j \in\{1,2\}$ such that $e_{j}^{i} \in M_{2}$. Hence, $M_{2}$ intersects $C_{1}$ and $C_{2}$ and so $\overline{M_{1}\cup M_{2}}$ is bipartite. \hfill $\square$\\

Using the same idea as in Proof 2 of Theorem \ref{Theorem S4 Oddness 2}, we also prove that the \mbox{$S_{4}$-Conjecture} is true for graphs having oddness four. 
\begin{theorem}\label{TheoremOddness4}
Let $G$ be a bridgeless cubic graph of oddness four. Then, $G$ has an \mbox{$S_{4}$-pair}. 
\end{theorem}
\begin{proof}
Let $M_{1}$ be a minimal perfect matching of $G$, and let $C_{1}, C_{2}, C_{3}$ and $C_{4}$ be the four odd cycles in $\overline{M_{1}}$. By Remark \ref{Remark3cuts}, there exists a perfect matching $N$ of $G$ such that if $G$ has any 3-cuts, then $N$ intersects every 3-cut of $G$ in one edge. Moreover, for every $i\in\{1,\ldots,4\}$, there exists at least a pair of adjacent edges $e_{1}^{i}$ and $e_{2}^{i}$ belonging to $C_{i}\cap\overline{N}$. We define the vector $c\in \mathbb{R}^{\vert E(G)\vert}$ such that
$$c(e)= \left\{
\begin{array}{cl}
1 & $if $ e\in \cup_{i=1}^{4}\{e_{1}^{i}, e_{2}^{i}\},\\
0 & $otherwise.$
\end{array}\right.$$
We also define the vector $w\in \mathbb{R}^{\vert E(G)\vert}$ as follows:
$$w(e)= \left\{
\begin{array}{cl}
\sfrac{1}{5} & $if $ e \in N,\\
\sfrac{2}{5} & $otherwise.$
\end{array}\right.$$
The vector $w$ is clearly a fractional perfect matching of $G$ because, in particular, $N$ intersects every 3-cut in one edge and so $w(X)\geq 1$ for each \mbox{odd-cut} $X$ of $G$. Hence, by Lemma \ref{LemmaFractionalPM}, there exists a perfect matching $M_{2}$ such that $c\cdot\chi^{M_{2}}\geq c\cdot w,$ which implies that $$\vert\cup_{i=1}^{4}\{e_{1}^{i}, e_{2}^{i}\} \cap M_{2} \vert\geq \sfrac{2}{5}\times 2 \times 4 = \sfrac{16}{5}>3.$$
Therefore, for each $i\in\{1,\ldots,4\}$, there exists exactly one $j \in\{1,2\}$ such that $e_{j}^{i} \in M_{2}$. Hence, $M_{2}$ intersects $C_{1}, C_{2}, C_{3}$ and $C_{4}$ and so $\overline{M_{1}\cup M_{2}}$ is bipartite.
\end{proof}
As the above proofs show us, extending results with respect to the \mbox{$S_{4}$-Conjecture} is easier than in the case of the \mbox{Fan-Raspaud} Conjecture and this is why we believe that a proof of the $S_4$-conjecture could be a first feasible step towards a solution of the \mbox{Fan-Raspaud} Conjecture. For graphs having oddness at least six we are not able to prove the existence of an \mbox{$S_{4}$-pair} and we wonder how many perfect matchings we need such that the complement of their union is bipartite. In the next proposition we use the technique used in Theorem \ref{TheoremOddness4} and show that given a bridgeless cubic graph $G$, if $\omega(G)\leq 5^{k-1}-1$ for some positive  integer $k$, then there exist $k$ perfect matchings such that the complement of their union is bipartite. Note that for $k=2$ we obtain $\omega(G)\leq 4$.
\begin{prop}\label{Proposition ni lowerbound}
Let $G$ be a bridgeless cubic graph and let $\mathcal{C}$ be a collection of disjoint odd cycles in $G$ such that $|\mathcal{C}|\leq 5^{k-1}-1$ for some positive integer $k$. Then, there exist $k-1$ perfect matchings of $G$, say $M_1,\ldots,M_{k-1}$, such that for every $C \in \mathcal{C}$, there exists $j\in\{1,\ldots,k-1\}$ for which $C \cap M_j \neq \emptyset$. Moreover, if $\omega(G)\leq 5^{k-1}-1$, then there exist $k$ perfect matchings such that the complement of their union is bipartite.
\end{prop}
\begin{proof}
We proceed by induction on $k$. 
For $k=1$, the assertion trivially holds since $\mathcal{C}$ is the empty set. Assume the result is true for some $k\geq 1$ and consider $k+1$. Let $C_{1}, C_{2}, \ldots, C_{t}$, with $t \leq 5^{k}-1$, be the odd cycles of $G$ in $\mathcal{C}$. Let $N$ be a perfect matching of $G$ which intersects every $3$-cut of $G$ once. For every $i\in\{1,\ldots,t\}$, there exists at least a pair of adjacent edges $e_{1}^{i}$ and $e_{2}^{i}$ belonging to $C_{i}\cap \overline{N}$. We define the vector $c\in \mathbb{R}^{|E(G)|}$ such that
$$c(e)= \left\{
\begin{array}{cl}
1 & $if $ e\in \cup_{i=1}^{t}\{e_{1}^{i}, e_{2}^{i}\},\\
0 & $otherwise.$
\end{array}\right.$$
We also define the vector $w\in \mathbb{R}^{\vert E(G)\vert}$ as follows:
$$w(e)= \left\{
\begin{array}{cl}
\sfrac{1}{5} & $if $ e \in N,\\
\sfrac{2}{5} & $otherwise.$
\end{array}\right.$$
As in the proof of Theorem \ref{TheoremOddness4}, $w$ is a fractional perfect matching of $G$ and by Lemma \ref{LemmaFractionalPM} there exists a perfect matching $M_{k}$ such that $c\cdot\chi^{M_{k}}\geq c\cdot w$. This implies that $$\vert\cup_{i=1}^{t}\{e_{1}^{i}, e_{2}^{i}\} \cap M_{k} \vert\geq 2 \times \sfrac{2}{5} \times t.$$ 
Let $\mathcal{C}'$ be the subset of $\mathcal{C}$ which contains the odd cycles of $\mathcal{C}$ with no edge of $M_{k}$. Then, $|\mathcal{C}'| \leq |\mathcal{C}| - \frac{4}{5} t = t - \frac{4}{5}t = \frac{t}{5} \leq 5^{k-1} - \frac{1}{5}$, and so $|\mathcal{C}'| \leq 5^{k-1}-1$. By induction, there exist  $k-1$ perfect matchings of $G$, say $M_{1}, \ldots, M_{k-1}$, having the required property with respect to $\mathcal{C}'$. Therefore, $M_{1},\ldots, M_{k}$ intersect all odd cycles in $\mathcal{C}$.
The second part of the statement easily follows by considering $\mathcal{C}$ to be the set of odd cycles in the complement of a minimal perfect matching $M$ of $G$, since the union of $M$ with the $k-1$ perfect matchings which intersect all the odd cycles in $\mathcal{C}$ has a bipartite complement.
\end{proof}
\begin{rem}\label{Remark DestroyingOddCycles FracPM}
We note that with every step made in the proof of Proposition \ref{Proposition ni lowerbound}, one could update the weight $w$ of the edges using the methods presented in \cite{FractionalPM,MazzuoccoloCovering} which gives a slightly better upper bound for $\omega(G)$.
For reasons of simplicity and brevity, we prefer the present weaker version of Proposition \ref{Proposition ni lowerbound}.
\end{rem}

\section{Extension of the \mbox{\boldmath{$S_{4}$}-Conjecture} to larger classes of cubic graphs} 
\subsection{Multigraphs}\label{Section Multigraphs}
In this section we discuss natural extensions of some previous conjectures to bridgeless cubic multigraphs. We note that bridgeless cubic multigraphs cannot contain any loops. 
We will make use of the following standard operation on parallel edges, referred to as \emph{smoothing}. Let $G'$ be a bridgeless cubic multigraph. Let $u$ and $v$ be two vertices in $G'$ such that there are exactly two parallel edges between them.

\begin{figure}[h]
      \centering
      \includegraphics[width=0.56\textwidth]{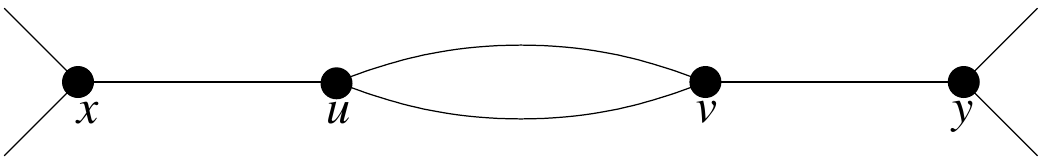}
      \caption{Vertices $x,u,v$ and $y$ in $G'$.}
      \label{Figure xuvy}
\end{figure}
Let $x$ and $y$ be the vertices adjacent to $u$ and $v$, respectively (see Figure \ref{Figure xuvy}). We say that we \emph{smooth $uv$} if we delete the vertices $u$ and $v$ from $G'$ and add an edge between $x$ and $y$ (even if $x$ and $y$ are already adjacent in $G'$). One can easily see that the resulting multigraph, say $G$, after smoothing $uv$ is again bridgeless and cubic. 

In what follows, we will say that a perfect matching $M$ of $G$ and a perfect matching $M'$ of $G'$  are \emph{corresponding} perfect matchings if the following holds: 

\[
 M =
  \begin{cases}
  M'\cup xy-\{xu,vy\} & $if $ xu\in M',\\
  M'-uv & $otherwise.$
  \end{cases}
\]

The following theorem can be easily proved by using smoothing operations.
\begin{theorem}\label{Theorem S4 Multigraphs}
The \mbox{$S_{4}$-Conjecture} is true if and only if every bridgeless cubic multigraph has an \mbox{$S_{4}$-pair}.
\end{theorem}

Now we show that Conjecture \ref{ConjectureS4**} can also be extended to multigraphs.

\begin{theorem}\label{Theorem FrequencyS4Multigraphs}
Let $i,j$ and $k$ be three integers in $\{0,1,2\}$ such that $i+j+k=2$ and let $w$ be a vertex in a bridgeless cubic multigraph $G'$. Then, the \mbox{$S_{4}$-Conjecture} is true if and only if $G'$ has an \mbox{$S_{4}$-pair} in which the edges incident to $w$ in a given order have frequencies $(i,j,k)$.
\end{theorem}
\begin{proof}
It suffices to assume that the \mbox{$S_{4}$-Conjecture} is true and only show the forward direction, by Theorem \ref{Theorem S4 Multigraphs}. Let $G'$ be a minimal counterexample and suppose it has some parallel edges. If $G'=C_{2,3}$ then the result clearly follows. So assume $G'\neq C_{2,3}$. Let $a,b$ and $c$ be the edges incident to $w$ such that the frequencies $(i,j,k)$ are to be assigned to $(a,b,c)$. We proceed by considering two cases: when $w$ has two parallel edges incident to it (Figure \ref{Figure FrequencyS4Multigraphs_A}) and otherwise (Figure \ref{Figure FrequencyS4Multigraphs_B}).
\begin{figure}[h]
      \centering
      \includegraphics[width=0.45\textwidth]{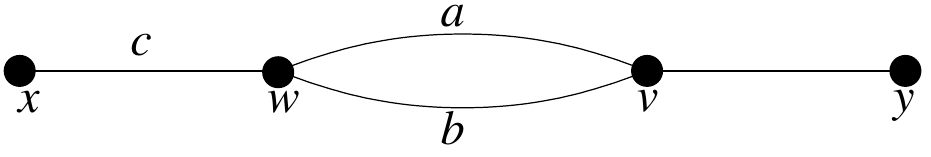}
      \caption{Case 1 in the proof of Theorem \ref{Theorem FrequencyS4Multigraphs}.}
      \label{Figure FrequencyS4Multigraphs_A}
\end{figure} 

\noindent\textbf{Case 1.} Let $G$ be the resulting multigraph after smoothing $wv$. By minimality of $G'$, $G$ has an \mbox{$S_{4}$-pair} (say $M_{1}$ and $M_{2}$) in which $\nu(xy)=k$. It is easy to see that a pair of corresponding perfect matchings in $G'$ give $\nu_{G'}(c)=\nu_{G'}(vy)=k$ and can be chosen in such a way such that $\nu_{G'}(a)=i$ and $\nu_{G'}(b)=j$, a contradiction to our initial assumption. Therefore, we must have Case 2.
\begin{figure}[h]
      \centering
      \includegraphics[width=0.35\textwidth]{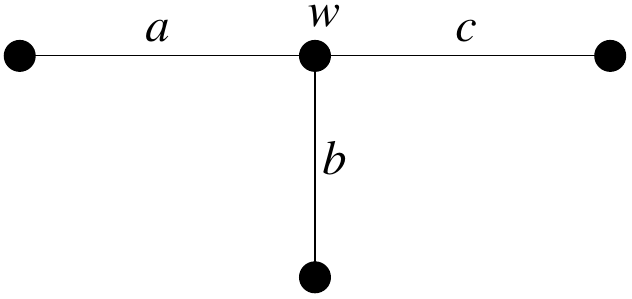}
      \caption{Case 2 in the proof of Theorem \ref{Theorem FrequencyS4Multigraphs}.}
      \label{Figure FrequencyS4Multigraphs_B}
\end{figure} 

\noindent\textbf{Case 2.} Let $G$ be the resulting multigraph after smoothing some parallel edge in $G'$ and let $a_{w}, b_{w}$ and $c_{w}$ be the corresponding edges incident to $w$ in $G$ after smoothing is done. In $G$, there exists an \mbox{$S_{4}$-pair} such that the frequencies of $(a_{w}, b_{w},c_{w})$ are equal to $(i,j,k)$. Clearly, the corresponding perfect matchings in $G'$ form an \mbox{$S_{4}$-pair} in which the frequencies of $(a,b,c)$ are $(i,j,k)$, a contradiction, proving Theorem \ref{Theorem FrequencyS4Multigraphs}.
\end{proof}
Using the same ideas as in Theorem \ref{Theorem S4 Multigraphs} and Theorem \ref{Theorem FrequencyS4Multigraphs} one can also state analogous results for the \mbox{Fan-Raspaud} Conjecture in terms of multigraphs.

\subsection{Graphs having bridges}\label{Section Graphs with bridges}
Since every perfect matching must intersect every bridge of a cubic graph, then the \break \mbox{Fan-Raspaud} Conjecture cannot be extended to cubic graphs containing bridges. The situation is quite different for the \mbox{$S_{4}$-Conjecture} as Theorem \ref{Theorem S4 for graphs with bridges} shows. By Errera's Theorem \cite{Errera} we know that if all the bridges of a connected cubic graph lie on a single path, then the graph has a perfect matching. We use this idea to show that there can be graphs with bridges that can have an \mbox{$S_{4}$-pair}. 

\begin{theorem}\label{Theorem S4 for graphs with bridges}
Let $G$ be a connected cubic graph having $k$ bridges, all of which lie on a single path, for some positive integer $k$. If the \mbox{$S_{4}$-Conjecture} is true, then $G$ admits an \mbox{$S_{4}$-pair}.
\end{theorem}
\begin{figure}[h]
      \centering
      \includegraphics[width=0.75\textwidth]{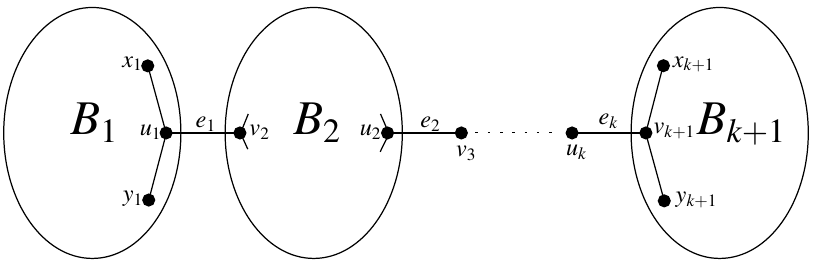}
      \caption{$G$ with $k$ bridges lying all on the same path.}
\end{figure}
\begin{proof}
Let $B_{1}, B_{2}, \ldots, B_{k+1}$ be the 2-connected components of $G$ and let $e_{1},\ldots, e_{k}$ be the bridges of $G$ such that $e_{i}=u_{i}v_{i+1}$ for each $i\in\{1,\ldots,k\}$, where $u_{i}\in V(B_{i})$ and $v_{i+1}\in V(B_{i+1})$. Let $x_{1}$ and $y_{1}$ be the two vertices adjacent to $u_{1}$ in $B_{1}$, and let $x_{k+1}$ and $y_{k+1}$ be the two vertices adjacent to $v_{k+1}$ in $B_{k+1}$. Let $B_{1}'=(B_{1}-u_{1})\cup x_{1}y_{1}$ and $B_{k+1}'=(B_{k+1}-v_{k+1})\cup x_{k+1}y_{k+1}$. Also, let $B_{i}'=B_{i}\cup v_{i}u_{i}$ for every $i\in \{2,\ldots,k\}$. Clearly, $B_{1}', \ldots, B_{k+1}'$ are bridgeless cubic multigraphs. Since we are assuming that the \mbox{$S_{4}$-Conjecture} holds, then, by Theorem \ref{Theorem S4 Multigraphs}, for every $i\in\{1,\ldots,k+1\}$, $B_{i}'$ has an \mbox{$S_{4}$-pair}, say $M_{1}^{i}$ and $M_{2}^{i}$. Using Theorem \ref{Theorem FrequencyS4Multigraphs}, we choose the \mbox{$S_{4}$-pair} in:
\begin{itemize}
\item $B_{1}'$, such that the two edges originally incident to $x_{1}$ (not $x_{1}u_{1}$) both have frequency $1$,
\item $B_{i}'$, for each $i\in \{2,\ldots,k\}$, such that $\nu_{B_{i}'}(v_{i}u_{i})=2$, and
\item $B_{k+1}'$, such that the two edges originally incident to $x_{k+1}$ (not $x_{k+1}v_{k+1}$) both have frequency $1$.
\end{itemize}
Let $M_{1}:=(\cup_{i=1}^{k+1}M_{1}^{i})\cup(\cup_{j=1}^{k}e_{j})-\cup_{l=2}^{k}v_{l}u_{l}$, and let $M_{2}:=(\cup_{i=1}^{k+1}M_{2}^{i})\cup(\cup_{j=1}^{k}e_{j})-\cup_{l=2}^{k}v_{l}u_{l}$. Then, $M_{1}$ and $M_{2}$ are an \mbox{$S_{4}$-pair} of $G$.
\end{proof}

Finally, we remark that there exist cubic graphs which admit a perfect matching however do not have an \mbox{$S_{4}$-pair}. For example, since the edges $u_{i}v_{i}$ in Figure \ref{Figure S4bridgesNOTTRUE} are bridges, then they must be in any perfect matching. Consequently, every pair of perfect matchings do not intersect the edges of the odd cycle $T$. This shows that it is not possible to extend the $S_4$-Conjecture to the entire class of cubic graphs.

\begin{figure}[h]
      \centering
      \includegraphics[width=0.4\textwidth]{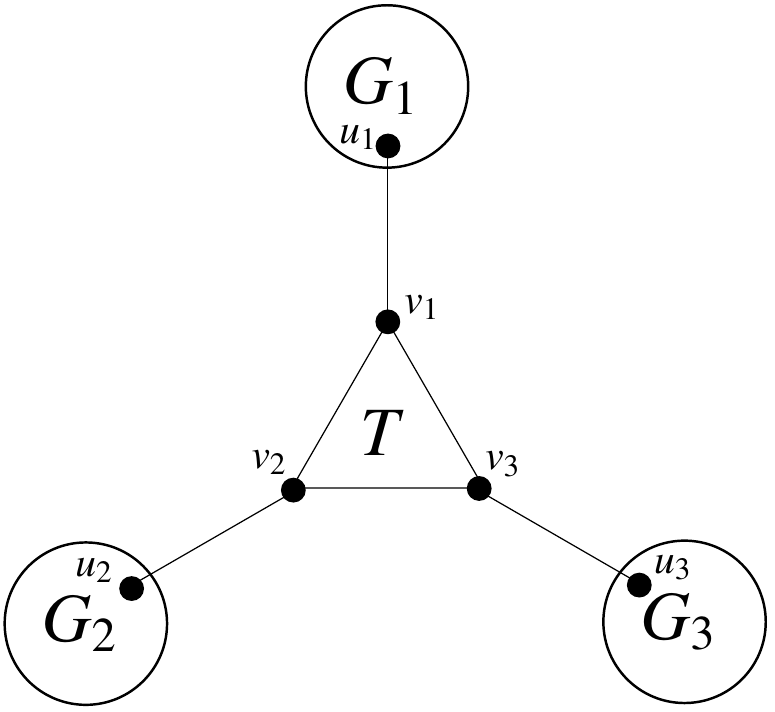}
      \caption{A cubic graph with bridges having no \mbox{$S_{4}$-pair}.}
      \label{Figure S4bridgesNOTTRUE}
\end{figure}

\section{Remarks and problems}

Many problems about the topics presented above remain unsolved: apart from asking if we can solve the \mbox{Fan-Raspaud} Conjecture and the \mbox{$S_{4}$-Conjecture} completely, or at least partially for higher oddness, we do not know which are those graphs containing bridges which admit an \mbox{$S_{4}$-pair} and we do not know either if the \mbox{$S_{4}$-Conjecture} is equivalent to Conjecture \ref{ConjectureOddCut}. 
Here we would like to add some other specific open problems.

For a positive integer $k$, we define $\omega_{k}$ to be the largest integer such that any graph with oddness at most $\omega_{k}$, admits $k$ perfect matchings with a bipartite complement. Clearly, for $k=1$, we have $\omega_1=0$, since the existence of a perfect matching of $G$ with a bipartite complement is equivalent to the $3$-edge-colourability of $G$. Moreover, the \mbox{$S_{4}$-Conjecture} is equivalent to  $\omega_{k}=\infty$, for $k\geq 2$, but a complete result to this is still elusive. Proposition \ref{Proposition ni lowerbound} (see also Remark \ref{Remark DestroyingOddCycles FracPM}) give a lower bound for $\omega_{k}$ and it would be interesting if this lower bound can be significantly improved. We believe that the following problem, weaker than the \mbox{$S_{4}$-Conjecture}, is another possible step forward. 
\begin{problem}
Prove the existence of a constant $k$ such that every bridgeless cubic graph admits $k$ perfect matchings whose union has a bipartite complement.
\end{problem}
It is also known that not every perfect matching can be extended to an \mbox{FR-triple} and neither to a \mbox{Berge-Fulkerson} cover, where the latter is a collection of six perfect matchings which cover the edge set exactly twice. We do not see a way how to produce a similar argument for \mbox{$S_{4}$-pairs} and so we also suggest the following problem.
\begin{problem}\label{Problem Given a Perfect Matching S4}
Establish whether any perfect matching of a bridgeless cubic graph can be extended to an \mbox{$S_{4}$-pair}.
\end{problem}

It can be shown that Problem \ref{Problem Given a Perfect Matching S4} is equivalent to saying that given any collection of disjoint odd cycles in a bridgeless cubic graph, then there exists a perfect matching which intersects all the odd cycles in this collection.


\begin{thebibliography}{99}
\bibitem{AbreuEtAl}
M. Abreu, T. Kaiser, D. Labbate and G. Mazzuoccolo,
Treelike Snarks,
\emph{Electron. J. Combin.} \textbf{23(3)} (2016), $\#$P3.54.

\bibitem{BonisoliCariolaro}
A. Bonisoli and D. Cariolaro,
Excessive factorizations of regular graphs,
in: A. Bondy, et al. (eds.), 
\emph{Graph Theory in Paris}, 
Birkh\"{a}user, Basel, 2007, 73--84.

\bibitem{Edmonds}
J. Edmonds,
Maximum matching and a polyhedron with $(0,1)$-vertices,
\emph{J. Res. Natl. Bur. Stand. Sec. B} \textbf{69} (1965), 125--130.

\bibitem{Errera}
A. Errera,
Une demonstration du theoreme de Petersen,
\emph{Mathesis} \textbf{36} (1922), 56--61.

\bibitem{EsperetLovaszPlummer}
L. Esperet, F. Kardo\v{s}, A.D. King, D. Kr\'{a}l and S. Norine,
Exponentially many perfect matchings in cubic graphs,
\emph{Adv. Math.} \textbf{227} (2011), 1646--1664.

\bibitem{EsperetMazzuoccolo}
L. Esperet and G. Mazzuoccolo,
On cubic bridgeless graphs whose edge-set cannot be covered by four perfect matchings,
\emph{J. Graph Theory} \textbf{77(2)} (2014), 144--157.

\bibitem{FanRaspaud}
G. Fan and A. Raspaud,
Fulkerson's Conjecture and circuit covers,
\emph{J. Combin. Theory Ser. B} \textbf{61(1)} (1994), 133--138.

\bibitem{FiolMazSteffen}
M.A. Fiol, G. Mazzuoccolo and E. Steffen,
Measures of Edge-Uncolorability of Cubic Graphs,
\emph{Electron. J. Combin.} \textbf{25(4)} (2018), $\#$P4.54.

\bibitem{FouquetVanherpe}
J.L. Fouquet and J.M. Vanherpe,
On the perfect matching index of bridgeless cubic graphs,
(2009), hal-00374313.

\bibitem{BergeFulkerson}
D.R. Fulkerson,
Blocking and anti-blocking pairs of polyhedra,
\emph{Math. Program.} \textbf{1(1)} (1971), 168--194.

\bibitem{StarProduct}
M. Funk, B. Jackson, D. Labbate and J. Sheehan,
2-Factor hamiltonian graphs,
\emph{J. Combin. Theory Ser. B} \textbf{87} (2003), 138--144.

\bibitem{HagglundSteffen}
J. H\"{a}gglund and E. Steffen,
Petersen-colorings and some families of snarks,
\emph{Ars Math. Contemp.} \textbf{7} (2014), 161--173.

\bibitem{HakMkr}
A. Hakobyan and V. Mkrtchyan,
$S_{12}$ and $P_{12}$--colorings of cubic graphs,
\emph{Ars Math. Contemp.} \textbf{17} (2019), 431--445.

\bibitem{Jaeger}
F. Jaeger,
Nowhere-zero flow problems,
in: L.W. Beineke, R.J. Wilson (eds.),
\emph{Selected Topics in Graph Theory 3},
Academic Press, London, 1988, 71--95.

\bibitem{JinSteffen}
L. Jin and E. Steffen,
Petersen Cores and the Oddness of Cubic Graphs,
\emph{J. Graph Theory} \textbf{84(2)} (2017), 109--120.

\bibitem{FractionalPM}
T. Kaiser, D. Kr\'{a}l and S. Norine,
Unions of Perfect Matchings in Cubic Graphs,
in: M. Klazar, et al. (eds.), 
\emph{Topics in Discrete Mathematics, Algorithms Combin. 26}, 
Springer, Berlin, Heidelberg, 2006, 225--230.

\bibitem{KaiserRaspaud}
T. Kaiser and A. Raspaud,
Perfect matchings with restricted intersection in cubic graphs,
\emph{European J. Combin.} \textbf{31} (2010), 1307--1315.

\bibitem{KMPRSS}
D. Kr\'{a}l, E. M\'{a}\v{c}ajov\'{a}, O. Pangr\'{a}c, A. Raspaud, J.S. Sereni and M. \v{S}koviera,
Projective, affine, and abelian colorings of cubic graphs,
\emph{European J. Combin.} \textbf{30} (2009), 53--69.

\bibitem{LinWang}
H. Lin and X. Wang,
Three matching intersection property for matching covered graphs,
\emph{DIMACS Ser. Discrete Math. Theoret. Comput. Sci.} \textbf{19(3)} (2017), $\#$16.

\bibitem{LovaszPlummer}
L. Lov\'{a}sz and M.D. Plummer,
\emph{Matching Theory}, first ed.,
Elsevier Science, Amsterdam, 1986.

\bibitem{MacajovaSkovieraOddCuts}
E. M\'{a}\v{c}ajov\'{a} and M. \v{S}koviera,
Fano colourings of cubic graphs and the Fulkerson Conjecture,
\emph{Theoret. Comput. Sci.} \textbf{349} (2005), 112--120.

\bibitem{MacajovaSkovieraOddness2}
E. M\'{a}\v{c}ajov\'{a} and M. \v{S}koviera,
Sparsely intersecting perfect matchings in cubic graphs,
\emph{Combinatorica} \textbf{34(1)} (2014), 61--94.

\bibitem{MazzuoccoloCovering}
G. Mazzuoccolo,
Covering a cubic graph with perfect matchings,
\emph{Discrete Math.} \textbf{313} (2013), 2292--2296.

\bibitem{MazzuoccoloS4}
G. Mazzuoccolo,
New conjectures on perfect matchings in cubic graphs,
\emph{Electron. Notes Discrete Math.} \textbf{40} (2013), 235--238.

\bibitem{MiaoWangZhang}
Z. Miao, X. Wang and C.Q. Zhang,
Unique Fulkerson coloring of Petersen minor-free cubic graphs,
\emph{European J. Combin.} \textbf{43} (2015), 165--171.

\bibitem{VahanPetersen}
V.V. Mkrtchyan,
A remark on the Petersen coloring conjecture of Jaeger,
\emph{Australas. J. Combin.} \textbf{56} (2013), 145--151.

\bibitem{Vahan}
V.V. Mkrtchyan and G.N. Vardanyan,
On two consequences of \mbox{Berge-Fulkerson} conjecture,
\emph{AKCE Int. J. Graphs Comb} (2019), \url{https:// doi.org/10.1016/j.akcej.2019.03.018}.

\bibitem{Petersen}
J. Petersen, 
Die Theorie der regul\"{a}ren graphs,
\emph{Acta Mathematica} \textbf{15} (1891), 193--220.

\bibitem{Samal}
R. \v{S}\'{a}mal,
New approach to Petersen coloring,
\emph{Electron. Notes Discrete Math.} \textbf{38} (2011), 755--760.

\bibitem{Steffen}
E. Steffen,
1-factor and cycle covers of cubic graphs,
\emph{J. Graph Theory} \textbf{78(3)} (2015), 195--206.

\bibitem{ZhuTangZhang}
Q. Zhu, W. Tang and C.Q. Zhang,
Perfect matching covering, the Berge-Fulkerson conjecture, and the Fan-Raspaud conjecture,
\emph{Discrete Appl. Math.} \textbf{166} (2014), 282--286.
\end{thebibliography}
\end{document}